\documentclass[reqno]{amsart}
\usepackage{amsfonts}
\usepackage{amssymb}
\usepackage{amsthm}
\usepackage{amsmath}
\usepackage{mathrsfs}
\usepackage{color}
\usepackage{enumerate}
\usepackage[numbers,sort&compress]{natbib}
\usepackage{pdfsync}
\usepackage{esint}
\usepackage{graphicx}
\usepackage{float}
\usepackage{caption}
\usepackage{subfigure}
\usepackage{enumitem}
\usepackage{appendix}
\usepackage{bbm}

\allowdisplaybreaks

\usepackage{tikz} 
\usetikzlibrary{calc}
\usetikzlibrary{intersections}
\usetikzlibrary{patterns}

\usepackage[colorlinks,
linkcolor=black,       
anchorcolor=blue,      
citecolor=blue,        
]{hyperref}

\makeatother

\numberwithin{equation}{section}

\newtheorem{theorem}{Theorem}[section]
\newtheorem{definition}[theorem]{Definition}

\newtheorem{lemma}[theorem]{Lemma}

\numberwithin{equation}{section}
\renewcommand{\P}{\mathbb{P}}
\theoremstyle{remark}
\newtheorem{remark}[theorem]{Remark}

\newcommand{\R}{{\mathbb R}}

\def\9{{\infty}}

\def\bbp{{\mathbb{P}}}
\def\({\left(}
\def\){\right)}
\def\<{\left<}
\def\>{\right>}

\newcommand{\one}{\mathbbm{1}}

\newcommand{\imu}{\mathrm{i}}
\newcommand{\dd}{\,\mathrm{d}}
\newcommand{\drm}{\mathrm{d}}
\renewcommand{\epsilon}{\varepsilon}

\renewcommand{\Re}{\operatorname{Re}}

\newcommand{\N}{{\mathbb{N}}}

\newcommand{\PP}{{\mathbb{P}}}

\begin{document}
	
	\title[Regularization by noise for critical NLS]{Regularization by noise for the energy- and mass-critical nonlinear Schr{\"o}dinger equations}
	
	\author{Martin Spitz}
	\address[Martin Spitz]{Fakult\"at f\"ur Mathematik,
		Universit\"at Bielefeld, D-33501 Bielefeld, Germany}
	\email{mspitz@math.uni-bielefeld.de}
	\thanks{}
	
	\author{Deng Zhang}
	\address[Deng Zhang]{School of Mathematical Sciences, MOE-LSC,
		CMA-Shanghai, Shanghai Jiao Tong University, China.}
	\email{dzhang@sjtu.edu.cn}
	\thanks{}
	
	\author{Zhenqi Zhao}
	\address[Zhenqi Zhao]{School of Mathematical Sciences, Shanghai Jiao Tong University, China.}
	\email{semimartingale@sjtu.edu.cn}
	\thanks{}

	\begin{abstract}
		In this article we prove a regularization by noise phenomenon for the energy-critical and mass-critical nonlinear Schr{\"o}dinger equations. We show that for any deterministic data, the probability that the corresponding solution exists globally and scatters goes to one as the strength of the non-conservative noise goes to infinity. The proof relies on the rescaling transform and a new observation on the rapid uniform decay of geometric Brownian motions after short time.
	\end{abstract}

	\maketitle
	
	\section{Introduction}

	We consider the stochastic nonlinear Schr{\"o}dinger equation (SNLS) 
	\begin{equation} 	\label{eq:StochasticNLS}
		\imu \drm X + \Delta X \drm t = \lambda |X|^{\alpha-1} X \drm t - \imu \mu X \drm t + \imu \sum_{k=1}^\infty X \phi_k\drm \beta_k(t)	
	\end{equation}
	in the energy-critical case $\alpha=1+4/(d-2)$ in dimensions $d \geq 3$ and in the mass-critical case 
	$\alpha = 1+ 4/d$ 
	in dimensions $d\geq 1$. 
	Here $\lambda=1$  or $\lambda=-1$, corresponding to the defocusing and focusing case, respectively. The stochastic term is taken in the sense of It\^o, 
	where $(\beta_k)$ are real-valued Brownian motions on a probability space $(\Omega, \mathscr{F}, \P)$ with normal, in particular right-continuous, filtration $(\mathscr{F}_t)_{t\geq 0}$, and 
	$(\phi_k)$ are complex numbers
	such that 
	\begin{align*} 
		\mu=\frac{1}{2}\sum_{k=1}^\infty
		|\phi_k|^2 <\infty. 
	\end{align*}
	We note that in the conservative case, i.e., in the case where all $\phi_k$ have vanishing real part, the stochastic part $- \imu \mu X \drm t + \imu \sum_{k=1}^\infty X \phi_k\drm \beta_k(t)$ equals the Stratonovich noise $ \imu \sum_{k=1}^\infty X \phi_k \circ \drm \beta_k(t)$, which motivates the form of the stochastic part in~\eqref{eq:StochasticNLS}. Furthermore, while this infinite sum can be mathematically reduced to a single one-dimensional complex Brownian motion -- as we will demonstrate later in this article -- we retain this formulation to naturally model systems driven by an arbitrary number of independent random perturbations.
	
	In this work we prove that in the non-conservative case, the probability that solutions to~\eqref{eq:StochasticNLS} exist globally and scatter converges to one as $\|\Re(\phi_k)\|_{\ell^2_k}$ goes to infinity. In the focusing case this is in striking contrast to the deterministic situation where finite-time blow-up solutions exist. Our main results thus show that adding non-conservative noise to the focusing energy- or mass-critical nonlinear Schr{\"o}dinger equations (NLS) improves the behavior of the corresponding solutions (with high probability), a phenomenon generally referred to as regularization by noise.
	
	Stochastic nonlinear Schr{\"o}dinger equations are a model for random nonlinear waves, describing, e.g., the propagation of nonlinear dispersive waves in nonhomogeneous or random media. They also appear in the modelling of open quantum systems, thermal fluctuations in molecular aggregates, scattering of excitons in crystals, etc. In this context we note that, for solutions $X$ of~\eqref{eq:StochasticNLS}, the map $t \mapsto \|X(t)\|_2^2$ is a continuous martingale as required in the application to quantum measurement~\cite{BG09}. We refer to, e.g., \cite{BCIRG94, BCIRG95, BG09} for more information on the motivation and physical background of SNLS.

	Local well-posedness of~\eqref{eq:StochasticNLS} is well-known, see \cite{BRZ14, BRZ16, BD99, BD03, BM14}. 
	As in the deterministic case, for general large initial data, the dynamics of SNLS in the defocusing cases differs greatly from the dynamics in the focusing cases. In fact, after extensive research over several decades, it is now known that in the deterministic defocusing critical cases solutions to NLS exist globally and scatter. 
	See~\cite{Bou99, CKSTT08, RV07, V07} for the energy-critical case 
	and \cite{Dod12, Dod16, Dod16.2} for the mass-critical case. 
	Corresponding results have been established in the stochastic defocusing critical cases, see~\cite{FX23, FX21, Z23}. We also refer to~\cite{OO19} for SNLS with additive noise in the defocusing mass- and energy-critical cases.
	
	However, in the focusing case, singularity formation and soliton dynamics occur. 
	More precisely, in the deterministic case, for the energy-critical NLS the existence of finite-time blow-up solutions follows from the classical virial identity~\cite{Gl77}. Further blow-up and soliton dynamics have been established in~\cite{KM06} and~\cite{DM09}. We also refer to~\cite{Sc23} for a recent construction of finite time type II blow-up solutions above the ground state and the references therein for an overview over the known blow-up results in the energy-critical case.
	
	For the mass-critical NLS there exist pseudo-conformal blow-up solutions, which are unique in the class of $H^1$ blow-up solutions~\cite{Mer93}. 
	Furthermore, a solution with critical mass blowing up in finite or infinite time must be a soliton or the pseudo-conformal transformation of a soliton, see~\cite{Dod23, Dod24}. Above the critical mass threshold finite-time blow-up solutions of Bourgain-Wang type exist~\cite{BW97}.
	
	In the stochastic focusing case, it was proved in~\cite{BD02,BD05} that the multiplicative noise can accelerate the blow-up of solutions,  
	namely, 	push solutions to blow up at any small time with positive probability.  
	The quantitative construction of  critical mass blow-up solutions to SNLS was performed in~\cite{SZ23}. 
	Afterwards, 	the multi-bubble Bourgain-Wang type blow-up solutions, including both multi blow-up profiles and a regular profile, 
	were constructed in \cite{SZ24}.  
	See also \cite{SZ23.2} for stochastic solitons. 
	
	In this article we prove that the addition of sufficiently strong non-conservative multiplicative noise to the critical SNLS prevents such singular behavior and yields global scattering solutions with high  probability instead. Such \emph{regularization by noise phenomena} have been observed for various stochastic models and are investigated intensively in the field of stochastic partial differential equations (SPDEs). Seminal regularization by noise results include~\cite{KR05} for stochastic differential equations (SDEs), \cite{DFRV16} for SDEs on Hilbert spaces, \cite{FGP10} for transport equations, and~\cite{FL21} for fluid models.
	
	For dispersive equations, it was observed numerically
	that multiplicative noise affects the blow-up and soliton dynamics. 
	For instance, in \cite{DM02, DM02.2},  
	numerical experiments show that conservative noise, which conserves the mass pathwisely, has the effect of delaying blow-up, 
	while white noise can even prevent blow-up. On the other hand, 
	it was shown in~\cite{BRZ17} 
	that in the energy-subcritical case the non-conservative noise has the ability to prevent blow-up with high probability. 
	Afterwards, 
	it was proved in~\cite{HRZ19}, still in the energy-subcritical case, that this type of noise even improves the scattering behavior. 
	Recently, such regularization by noise effects have also been observed for the stochastic Zakharov system in dimensions three and four~\cite{HRSZ23, HRSZ24}. 
	
	We note that regularization by noise phenomena were also observed for the nonlinear Schr{\"o}dinger equation with random dispersion. In particular, it was shown that the mass-critical NLS with random dispersion is globally well-posed even for large $L^2$-data in the focusing case, see~\cite{DT11, Ro24} and the references therein. 
	We also refer to the recent work~\cite{BFMZ24} which proves that adding suitable superlinear noise leads to the global well-posedness of NLS with polynomial type nonlinearity for any initial data in $H^s(\mathbb{T}^d)$, $s>d/2$, and thus prevents blow-up in finite time.

	As discussed above, scattering from regularization by noise for SNLS was shown in~\cite{HRZ19} in the energy-subcritical case, but the energy-critical case remained open. Our first main result solves this problem, establishing scattering from regularization by noise in the energy-critical case.
	
	Here and in the following we use the notation $\|c\| = \|c\|_{\ell^2}$ for any $c \in \ell^2$ and $W(t) := \sum_{k = 1}^\infty \phi_k \beta_k(t)$.
	
	\begin{theorem}[Noise regularization for energy-critical SNLS]
		\label{thm:RegNoiseEnerCrit}
		Let $d \geq 3$, $\alpha = 1 + \frac{4}{d-2}$, $c_k = \Re(\phi_k)$ for $k \in \N$ such that $c \in \ell^2$. Let $X_0 \in H^1(\R^d)$. Then
		\begin{align*}
			\PP(\{X \text{ exists on } [0,\infty) \text{ and scatters forward in time}\}) \longrightarrow 1 \qquad \text{as } \|c\|  \rightarrow \infty,
		\end{align*}
		where $X$ denotes the unique solution of~\eqref{eq:StochasticNLS} with initial data $X(0) = X_0$ and ``scatters forward in time'' means that there exists $X_+ \in H^1(\R^d)$ such that
		\begin{align} \label{Noise-scatter-H1}
			\lim_{t \rightarrow \infty} \| e^{-\imu t \Delta} e^{\hat{\mu} t - W(t)} X(t) - X_+ \|_{H^1(\R^d)} = 0,
		\end{align}
		where $\hat{\mu} = \frac{1}{2}\sum_{k=1}^\infty (|\phi_k|^2 + \phi_k^2)$.
	\end{theorem} 
	
	We further prove that this regularization by noise phenomenon also occurs in the mass-critical case.
	
	\begin{theorem}[Noise regularization for   mass-critical SNLS]
		\label{thm:RegNoiseMassCrit}
		Let $d \geq 1$, $\alpha = 1 + \frac{4}{d}$, and $c$ as in Theorem~\ref{thm:RegNoiseEnerCrit}. Let $X_0 \in L^2(\R^d)$. We then have
		\begin{align}  \label{Noise-scatter-L2}
			\PP(\{X \text{ exists on } [0,\infty) \text{ and scatters forward in time}\}) \longrightarrow 1 \qquad \text{as } \|c\| \rightarrow \infty,
		\end{align}
		where $X$ denotes the unique solution of~\eqref{eq:StochasticNLS} with initial data $X(0) = X_0$ and ``scatters forward in time'' means that there exists $X_+ \in L^2(\R^d)$ such that
		\begin{align*}
			\lim_{t \rightarrow \infty} \| e^{-\imu t \Delta} e^{\hat{\mu} t - W(t)} X(t) - X_+ \|_{L^2(\R^d)} = 0,
		\end{align*}
		where $\hat{\mu}$ is as in 
		Theorem \ref{thm:RegNoiseEnerCrit}.
	\end{theorem}
	
	We mention that the analogous result in the mass-subcritical case, although not treated in the literature,	can be proved as Theorem 1.7 of~\cite{HRZ19}. Therefore, we focus on the new contribution which is the mass-critical case. 
	
	\begin{remark}
		We point out that Theorems \ref{thm:RegNoiseEnerCrit} and~\ref{thm:RegNoiseMassCrit} 
		hold for any initial data, in particular for initial data  for which deterministic solutions blow up or exhibit soliton dynamics in the focusing case. This unveils the regularization effect of noise on the long time dynamics of NLS. 
	\end{remark} 
	
	\begin{remark}
		One can also consider SNLS~\eqref{eq:StochasticNLS} with additional	temporal functions $(g_k(t))$ in the noise as in~\cite{HRZ19} and derive the noise regularization effect on blow-up and scattering by arguing in an analogous manner	as in the present work. For clarity of exposition, we refrained from doing so.
	\end{remark}

	\begin{remark}
		Our proof relies on a new observation for geometric Brownian motions. Applying the rescaling transform, the stochastic equation~\eqref{eq:StochasticNLS} is equivalent to a random NLS with a geometric Brownian motion in the nonlinearity, see~\eqref{eq:RandomNLS} below. In the subcritical case, using Strichartz estimates, one can estimate the geometric Brownian motion in $L^\theta_t$ for some $\theta < \infty$ and gain smallness from this factor, which is the crucial ingredient in the proof there. This strategy fails in the critical case.
		
		Instead we observe that the geometric Brownian motion decays uniformly and rapidly after the small time $\|c\|^{-1}$ with high probability, 
		where $\|c\|^2$ is exactly the quadratic variation of the noise $\sum_{k=1}^\infty c_k \beta_k$ in the geometric Brownian motion \eqref{eq:DefGeomBr} below at unit time. This allows us to construct global and scattering solutions in a two step procedure: Before the time  $\|c\|^{-1}$ we use smallness of the linear profile in a dispersive norm to construct a solution on $[0, \|c\|^{-1}]$ with a \emph{deterministic} bound on the size of the solution provided $\|c\|$ is large enough. In a second step we then solve the equation on $[\|c\|^{-1}, \infty)$, exploiting scaling properties of Brownian motion and the fast uniform decay of the geometric Brownian motion after the small time  $\|c\|^{-1}$.
		
		We refer to Sections~\ref{Sec-Energy} and~\ref{Sec-Mass} for more detailed explanations of the proof.
	\end{remark} 
	
	Finally, we would like to mention that the non-conservative noise does not act as a monotone dissipative mechanism and is fundamentally distinct from deterministic damping. While deterministic damping forces an exponential decay of the mass for all times, the noise conserves the mass in expectation; the rapid decay we exploit is a strictly pathwise phenomenon relying on the delicate properties of  Brownian motions, 
    e.g., the law of the iterated logarithm and the time-change for martingales.

	\section{Probabilistic tools}
	\label{sec:Rescaling}
	
	We first give the definition of a solution of~\eqref{eq:StochasticNLS}. As we are studying $H^1$-solutions in the energy critical case and $L^2$-solutions in the mass critical case, we set $s(\alpha) = 1$ if $\alpha = 1 + \frac{4}{d-2}$ and $s(\alpha) = 0$ if $\alpha = 1 + \frac{4}{d}$ to give the definition for both cases simultaneously. Moreover, we make the convention $H^0(\R^d) = L^2(\R^d)$.
	
	\begin{definition}
		\label{def:SolSNLS}
		Let $d \geq 1$ and $\alpha = 1 + \frac{4}{d}$ or $d \geq 3$ and $\alpha = 1 + \frac{4}{d-2}$. Fix $T\in (0,\infty)$. We say that $X$ is a probabilistic strong solution of~\eqref{eq:StochasticNLS} on $[0,\tau]$,
		where $\tau \in (0,T]$ is an $\{\mathscr{F}_t\}$-stopping time,
		if $X$ is an $H^{s(\alpha)}$-valued
		$\{\mathscr{F}_t\}$-adapted process which belongs to $C([0,\tau],H^{s(\alpha)})$
		and satisfies $\mathbb{P}$-a.s.
		for any $t\in [0,\tau]$,
		\begin{equation}   \label{eq:SNLSDef}
			X(t) = X_0+ \int_0^t \imu \Delta X \dd s - \imu \int_0^t \lambda |X|^{\alpha - 1} X \dd s - \int_0^t \mu X \dd s + \sum_{k = 1}^\infty \int_0^t X \phi_k \dd \beta_k(s),  
		\end{equation}
		as equations in $H^{-2 + s(\alpha)}(\R^d)$.
		
		Given an $\{\mathscr{F}_t\}$-stopping time $\tau^*$, we also call $X$ a probabilistic strong solution of~\eqref{eq:StochasticNLS}
		on $[0,\tau^*)$ if $X$ is an $\{\mathscr{F}_t\}$-adapted process belonging to $C([0,\tau^*),H^{s(\alpha)})$ such that for any $T \in (0,\infty)$ and any $\{\mathscr{F}_t\}$-stopping time $\tau < \tau^*$, $X$ is a probabilistic strong solution of~\eqref{eq:StochasticNLS} on $[0,\tau \wedge T]$.
	\end{definition}	
	
	We first use the rescaling or Doss-Sussmann type transformation  
	\begin{align}
		u(t) = e^{\widehat{\mu} t - W(t)} X(t), 
	\end{align}  
	where $\hat{\mu} = \frac{1}{2}(\sum_{k=1}^\infty |\phi_k|^2 + \phi_k^2)$ 
	is as in Theorem \ref{thm:RegNoiseEnerCrit} and $W(t) = \sum_{k = 1}^\infty \phi_k \beta_k(t)$, 
	to reduce the stochastic equation \eqref{eq:StochasticNLS} 
	to the random NLS 
	\begin{equation}
		\label{eq:RandomNLS}
		\imu \partial_t u + \Delta u = \lambda h_c |u|^{\alpha - 1} u,
	\end{equation}
	where $h_c$ is the geometric Brownian motion
	\begin{align}
		\label{eq:DefGeomBr}
		h_c = e^{(\alpha - 1)(\sum_{k=1}^\infty c_k \beta_k(t) - (\sum_{k=1}^\infty c_k^2 ) t)}
	\end{align}
	with $c_k = \Re(\phi_k)$. 
	
	A heuristic application of It{\^o}'s product formula indicates that one can transform the stochastic equation~\eqref{eq:StochasticNLS} into the random NLS~\eqref{eq:RandomNLS}, see Section~3 in~\cite{BRZ14}. The rigorous proofs of the equivalence of~\eqref{eq:StochasticNLS} and~\eqref{eq:RandomNLS} are given in Lemma~6.1 in~\cite{BRZ14} and Lemma~2.4 in~\cite{BRZ16}. Consequently, one only needs to prove Theorems~\ref{thm:RegNoiseEnerCrit} and~\ref{thm:RegNoiseMassCrit} for the transformed random NLS~\eqref{eq:RandomNLS}.

	One key observation employed in these proofs is that the geometric Brownian motion $h_c$ has rapid uniform decay after the small time $\|c\|^{-1}$ as explained in the following lemma.
	
	\begin{lemma}
		\label{lem:GeomBrDec}
		Let $h_c$ denote the geometric Brownian motion as in \eqref{eq:DefGeomBr} 
		and let $c=(c_1, c_2, \ldots) \in \ell^2$. 
		Then, for any $\varepsilon>0$, we have 
		\begin{align*}
			\PP(\{\|h_c\|_{L^\infty_t([\|c\|^{-1},\infty))}>\varepsilon\})\longrightarrow 0 \qquad \text{as}\  \|c\| \rightarrow \infty.
		\end{align*}
	\end{lemma}
	
	\begin{proof} 
		We first note that 
		$M(t):= \sum_{k=1}^\infty c_k \beta_k(t)$, $t\geq 0$, 
		is a continuous martingale  
		and has the quadratic variation 
		$\<M\>(t) = (\sum_{k=1}^\infty c_k^2) t$.  
		Hence, 
		by the time-change for martingales (see, e.g. 
		\cite[Theorem 19.4]{Ka21}), 
		there exists a Brownian motion $\widetilde{B}$ 
		such that 
		$M(t) = \widetilde{B}\circ \<M\>(t)$, 
		$t\geq 0$, $\mathbb{P}$-a.s.
		
		This identity yields that  
		\begin{align} \label{h-prob}
			\PP(\{\|h_c\|_{L^\infty_t([\|c\|^{-1},\infty))}>\varepsilon\})
			&= \PP\bigg(\bigg\{\|e^{(\alpha -1)(\sum_{k=1}^{\infty} c_k\beta_k(t) - \|c\|^2 t) }\|_{L^\infty_t([\|c\|^{-1},\infty))}>\varepsilon\bigg\}\bigg)\nonumber \\
			&= \PP\bigg(\bigg\{\|e^{(\alpha -1)(\tilde{B}(\|c\|^2 t) - \|c\|^2 t) }\|_{L^\infty_t([\|c\|^{-1},\infty))}>\varepsilon\bigg\}\bigg)\nonumber \\
			&=\PP\bigg(\bigg\{\|e^{(\alpha - 1)(\widetilde{B}(t) 
				- t)}\|_{L^\infty_t([\|c\|,\infty))}>
			\varepsilon 
			\bigg\}\bigg). 
		\end{align} 
		
		In view of the law of the iterated logarithm  
		\begin{align}\label{eq:LawItLog}
			\limsup\limits_{t\to \infty} \frac{\widetilde{B}(t)}{\sqrt{2t \log\log t}} =1, \ \
			\liminf\limits_{t\to \infty} \frac{\widetilde{B}(t)}{\sqrt{2t \log\log t}} = -1, \ \ \bbp\text{-a.s.}, 
		\end{align} 
		we thus infer that the right-hand side of \eqref{h-prob} converges to zero as $\|c\| \rightarrow \infty$. 
	\end{proof}

	\section{The energy-critical case}
	\label{Sec-Energy}
	
	We first treat the random NLS~\eqref{eq:RandomNLS} in the energy-critical case  $\alpha= 1+4/(d-2)$, where $d\geq 3$. 
	More generally, we consider the equation 
	\begin{equation} \label{hNLS-Energy}
		\imu \partial_t u + \Delta u = \lambda h |u|^{\frac{4}{d-2}} u,
	\end{equation}
	where $h$ is a continuous function on $\R^+ := [0,\infty)$. 
	
	We say that a pair $(q,p)$ is Schr{\"o}dinger admissible if
	\begin{align*}
		\frac{2}{q} + \frac{d}{p} = \frac{d}{2}, 
		\quad 
		(d,q,p) \neq (2,2,\infty).
	\end{align*}

	\begin{lemma}[Strichartz estimates \cite{KT98}]
		\label{lem:StrichartzEstimates}
		Let $I \subseteq \R$ be an interval and $t_0 \in I$. Let $(q,p)$ and $(\tilde{q}, \tilde{p})$ be Schr{\"o}dinger admissible pairs. 
		Then, one has 
		\begin{align*}
			\|e^{\imu(t-t_0)\Delta} f\|_{L^q_I L^p_x} &\lesssim \| f \|_{L^2_x}, \\
			\Big\| \int_{t_0}^t e^{\imu(t-s)\Delta} g(s) \dd s \Big\|_{L^q_I L^p_x} &\lesssim \|g\|_{L^{\tilde{q}'}_I L^{\tilde{p}'}_x}, 
		\end{align*} 
		where the implicit constants are independent of $I$. 
	\end{lemma}
	
	In the energy-critical case,	
	we will use the Schr\"odinger admissible pair 
	\begin{equation*}
		(q,p)= \Big(\frac{2d}{d-2}, \frac{2d^2}{d^2-2d+4}\Big).
	\end{equation*} 
	For any time interval $I\subseteq \mathbb{R}$, 
	we define the spaces $S(I)$ and $S^1(I)$ as the closure of all tempered distributions with respect to the norms 
	\begin{align*}
		\|u\|_{S(I)} := \|u\|_{L^q(I; L^p( \R^d))},\quad \|u\|_{S^1(I)} := \|u\|_{L^q(I;  W^{1,p}(\R^d))},
	\end{align*}
	respectively.
	We further set 
	\begin{align*}
		\|u\|_{N(I)} := \|u\|_{L_t^{q'}(I; L_x^{p'}(\R^d))},  
	\end{align*}
	where $(q',p')=(\frac{2d}{d+2}, \frac{2d^2}{d^2+2d-4})$ 
	is the dual pair of $(q,p)$, and define the space $N(I)$ as the closure of the tempered distributions with respect to this norm.
	
	We record the estimates for the energy-critical nonlinearity we are going to use.	
	\begin{lemma}[Estimates for the nonlinearity] \label{Lem-Nonl-Energy} 
		Let $I \subseteq \R^+$ be an interval and $F(u):=|u|^{\frac{4}{d-2}}u$. 
		Then, one has 
		\begin{align}
			&\|F(u)\|_{N(I)} \lesssim \|u\|_{S(I)}\|\nabla u\|^{\frac{4}{d-2}}_{S(I)},  \label{prop:EnerNonStr}\\
			&\|F(u)-F(v)\|_{N(I)}\lesssim \big(\|\nabla u\|_{S(I)}^{\frac{4}{d-2}}+\|\nabla v\|_{S(I)}^{\frac{4}{d-2}}\big)\|u-v\|_{S(I)}, \label{prop:EnerNonDiff}\\
			&\|\nabla F(u)\|_{N(I)} \lesssim \|\nabla u\|^{1+\frac{4}{d-2}}_{S(I)}.  \label{prop:EnerNablaNonStr}
		\end{align}
	\end{lemma}
	\begin{proof}
		It is well known that using the complex derivatives $\partial_z = \frac{1}{2}(\partial_x - \imu \partial_y)$ and $\partial_{\overline{z}} = \frac{1}{2}(\partial_x + \imu \partial_y)$ one can easily show the pointwise estimates
		\begin{equation}
			\label{eq:DifferenceEstPointwise}
			|F(u) - F(v)| \leq C (|u|^{\frac{4}{d-2}} + |v|^{\frac{4}{d-2}})|u-v|
		\end{equation}
		and $|\nabla F(u)| \leq C |u|^{\frac{4}{d-2}} |\nabla u|$, see e.g. Remark~2.3 in~\cite{KM06}. The estimates~\eqref{prop:EnerNonStr} to~\eqref{prop:EnerNablaNonStr} then follow from H{\"o}lder's inequality and Sobolev embedding.
	\end{proof}

	The following result treats the well-posedness of~\eqref{hNLS-Energy} 
	on small time intervals. 
	
	\begin{lemma} 
		\label{la:SmallIniEner}
		Let $T > 0$ and $u_0 \in H^1(\R^d)$. Consider the initial value problem 
		\begin{equation}\label{eq:GenEnerEq}
			\left\{\aligned
			\imu \partial_t u + \Delta u &= \lambda h |u|^{\frac{4}{d-2}}u\\
			u(0)&=u_0	
			\endaligned
			\right.
		\end{equation}
		on $I=[0,T]$, where $h \in L^\infty(I)$. Let $A > 0$ such that $\|h\|_{L^\infty(I)} \leq A$ and $\delta=\delta(A)>0$ be so small that 
		\begin{align}\label{eq:AssDeltaEner}
			4C A (2\delta)^{\frac{4}{d-2}}\leq 1,
		\end{align}
		where $C$ is a large fixed generic constant. 
		Then, 
		for any initial datum $u_0$ 
		satisfying 
		\begin{align} \label{initial-small-H1}
			\|  e^{\imu t \Delta}u_0\|_{S^1(I)}\leq \delta,
		\end{align} 
		there exists a unique solution $u\in C([0, T], H^1(\R^d)) \cap S^1(I)$ to  $(\ref{eq:GenEnerEq})$. Moreover, the solution $u$ satisfies
		\begin{align*}
			\|u\|_{C(I; H^1(\R^d))} \leq \| u_0 \|_{H^1(\R^d)} + 1.
		\end{align*}
	\end{lemma}
	
	\begin{proof} 
		We consider the ball
		\begin{align*}
			B_1:=\{ u\in S^1(I): \|u\|_{S^1(I)}\leq 2\delta\}
		\end{align*}
		endowed with the metric $d(u,v):=\|u-v\|_{S(I)}$. Note that $B_1$ is closed and thus complete with this metric.
		Define the map $\Phi(u_0; \cdot)$ on $B_1$ by 
		\begin{align*}
			\Phi(u_0; u)(t) :=e^{\imu t\Delta}u_0 
			- \imu \lambda \int_0^t e^{\imu (t-s) \Delta} 
			(h |u|^{\frac{4}{d-2}}u) \dd s
		\end{align*}
		for any $u\in B_1$. 
		
		Then, by Strichartz estimates, Lemma \ref{Lem-Nonl-Energy},  
		\eqref{eq:AssDeltaEner} 
		and~\eqref{initial-small-H1}, 
		we have
		\begin{align} 
			\| \Phi(u_0; u)\|_{S^1(I)}
			&\leq \|e^{\imu t\Delta }u_0\|_{S^1(I)}+ C\|h\|_{L^\infty(I)}  \| u\|_{S^1(I)}^{1+\frac{4}{d-2}} \nonumber\\ 
			&\leq \delta+ C\|h\|_{L^\infty(I)}(2\delta)^{1+\frac{4}{d-2}} 
			\leq 2\delta, \label{eq:EnCrSelfMap}
		\end{align}
		as well as
		\begin{align*}
			\|\Phi(u_0; u)-\Phi(u_0;v)\|_{S(I)}
			&\leq C  \|h\|_{L^\infty(I)} (\|\nabla u\|_{S(I)}^{\frac{4}{d-2}} + \|\nabla v\|_{S(I)}^{\frac{4}{d-2}})\|u-v\|_{S(I)} \\ 
			&\leq 2C\|h\|_{L^\infty(I)}(2\delta)^{\frac{4}{d-2}}\|u-v\|_{S(I)} \leq \frac 12 \|u-v\|_{S(I)}.
		\end{align*}  
		Hence, $\Phi(u_0; \cdot)$ is a contractive self-map on $B_1$.  In particular,  
		$\Phi(u_0;\cdot)$ has a unique fixed point $u = \Phi(u_0; u)$ in $B_1$. Strichartz estimates also show as in~\eqref{eq:EnCrSelfMap} that $u \in C(I; H^1(\R^d))$.  
		It follows that $u$ solves \eqref{eq:GenEnerEq} on $[0,T]$. The uniqueness can be proved by similar arguments. Finally, using Strichartz estimates and arguing as above, we obtain
		\begin{align*}
			\|\Phi(u_0; u)\|_{C(I;  H^1(\R^d))}
			&\leq \|u_0\|_{H^1(\R^d)}+C\|h\|_{L^\infty(I)}\|\langle \nabla\rangle (|u|^{\frac{4}{d-2}}u)\|_{N(I)}\nonumber\\
			&\leq \|u_0\|_{H^1(\R^d)}+C\|h\|_{L^\infty(I)}\| u\|_{S^1(I)}^{1+\frac{4}{d-2}}  \nonumber \\ 
			&\leq \|u_0\|_{H^1(\R^d)}+C\|h\|_{L^\infty(I)}(2\delta)^{1+\frac{4}{d-2}} \leq \|u_0\|_{H^1(\R^d)} + 1.
		\end{align*}
		Thus, the proof is complete.  
	\end{proof}

	The following lemma shows that equation~\eqref{hNLS-Energy} is globally well posed 
	if $h$ is sufficiently small. 
	
	\begin{lemma}
		\label{la:SmallPerturEner}
		Let $T > 0$ and $u_T \in H^1(\R^d)$. Consider the initial value problem 
		\begin{equation}\label{eq:GenEnerEqLT}
			\left\{\aligned
			\imu \partial_t u + \Delta u &= \lambda h |u|^{\frac{4}{d-2}}u\\
			u(T)&=u_T	
			\endaligned
			\right.
		\end{equation}
		on $J=[T, \infty)$, where $h \in L^\infty(J)$.
		Let $E>1$ be such that $\|u_T\|_{H^1(\R^d)}\leq  E$ and $\varepsilon=\varepsilon(E)>0$ be so small that 
		\begin{align}\label{eq:AssVarepsilonEner}
			4C\varepsilon (2CE)^{\frac{4}{d-2}}\leq 1,
		\end{align}
		where $C$ is a large fixed generic constant.
		Further assume that
		\begin{align}  \label{h-small-H1}
			\|h\|_{L^\infty(J)}\leq  \varepsilon. 
		\end{align}

		Then, there exists a unique solution 
		$u\in C([T, \infty); H^1(\R^d)) \cap S^1([T,\infty))$ of~\eqref{eq:GenEnerEqLT}. Moreover, $u$ scatters at infinity, i.e., there exists $f_+ \in H^1(\R^d)$ such that  
		\begin{align}  \label{u-scatt-H1} 
			\lim_{t\to\infty} \big\|u(t)- e^{\imu t \Delta} f_+ \big\|_{H^1(\R^d)}=0.
		\end{align} 
	\end{lemma}
	
	\begin{proof}
		We consider the ball
		\begin{align*}
			B_2:=\{ u\in S^1(I): \|u\|_{S^1(I)} \leq 2CE\}
		\end{align*}
		equipped with the metric $d(u,v):=\|u-v\|_{S(I)}$. Note that $B_2$ is closed and thus complete with this metric.
		Define the map $\Phi(u_T; \cdot)$ on $B_2$ by  
		\begin{align*}
			\Phi(u_T;u)(t):=e^{\imu (t-T) \Delta}u_T 
			- \imu \lambda  \int_{T}^t e^{\imu (t-s)\Delta} (h|u|^{\frac{4}{d-2}}u)\dd s. 
		\end{align*}
		
		Then, 
		for any $u,v \in B_2$, Strichartz estimates, Lemma \ref{Lem-Nonl-Energy}, \eqref{eq:AssVarepsilonEner}, and~\eqref{h-small-H1} imply
		\begin{align*}
			\|\Phi(u_T;u)\|_{C(J; H^1(\R^d))} + \| \Phi(u_T;u)\|_{S^1(J)}
			&\leq
			C\|u_T\|_{H^1(\R^d)} + C\|h\|_{L^\infty(J)}  \| u\|_{S^1(J)}^{1+\frac{4}{d-2}} \nonumber \\ 
			&\leq CE +C\varepsilon(2CE)^{1+\frac{4}{d-2}}
			\leq 2 CE,
		\end{align*} 
		and
		\begin{align*}
			\|\Phi(u_T;u)-\Phi(u_T;v)\|_{S(J)}
			&\leq
			C\|h\|_{L^\infty(J)} (\|\nabla u\|_{S(J)}^{\frac{4}{d-2}} + \|\nabla v\|_{S(J)}^{\frac{4}{d-2}})\|u-v\|_{S(J)}  \nonumber \\ 
			&\leq 2C\varepsilon(2CE)^{\frac{4}{d-2}}\|u-v\|_{S(J)} \leq \frac 12 \|u-v\|_{S(J)}.
		\end{align*}
		
		It follows that $\Phi(u_T; \cdot)$ has a unique fixed point $u$ in $B_2$, 
		which solves~\eqref{eq:GenEnerEqLT}  on $[T,\infty)$ and is contained in $C(J; H^1(\R^d))$. By standard arguments one also obtains that $u$ is unique in $S^1(J)$.

		It remains to show the scattering property. 
		To this end, employing Strichartz estimates and Lemma~\ref{Lem-Nonl-Energy} again,
		we derive that for any $t_2 > t_1 \geq T$, 
		\begin{align*}
			\big\| e^{-\imu t_1\Delta}u(t_1)-e^{-\imu t_2\Delta}u(t_2)\big\|_{H^1(\R^d)} 
			&= \Big\| \int_{t_1}^{t_2} e^{-\imu s\Delta} (h|u|^{\frac{4}{d-2}}u)\Big\|_{H^1(\R^d)}\\
			&\leq \|h\|_{L^\infty([t_1,t_2])}\|u\|_{S([t_1,t_2])}\|u\|_{S^1([t_1,t_2])}^{\frac{4}{d-2}}\longrightarrow 0, 
		\end{align*} 
		as $t_1, t_2 \rightarrow \infty$ since $\|u\|_{S^1(J)}<\infty$. 
		Hence, 
		we infer that 
		$e^{-\imu t \Delta} u(t)$ converges in $H^1(\R^d)$ as $t \rightarrow \infty$
		and obtain the scattering behavior \eqref{u-scatt-H1} of the solution.   
	\end{proof}

	We are now ready to prove Theorem \ref{thm:RegNoiseEnerCrit}. 
	
	\medskip 
	{\bf Proof of Theorem $\ref{thm:RegNoiseEnerCrit}$.}  
	As mentioned above, using the rescaling transform one only needs to prove the corresponding assertion for the rescaled random NLS~\eqref{hNLS-Energy} in the energy-critical case. 
	
	For this purpose, 
	in view of the law of the iterated logarithm~\eqref{eq:LawItLog} and the continuity of Brownian motions, we infer that  $h := h_{e_1}$, where $e_1$ is the first unit vector in $\ell^2$, is bounded on $[0,\infty)$,  $\PP$-a.s., 
	i.e.,
	\begin{align*}
		\PP\Big(\bigcup_{n \in \N} \{\|h\|_{L^\infty([0,\infty))} \leq n\} \Big) = 1.
	\end{align*}
	In particular, for any $\eta>0$, there exists $A > 0$ such that
	\begin{align*}
		&\PP(\{\omega \in \Omega \colon \|h_c(\omega)\|_{L^\infty([0,\infty))} \leq A\}) \\
		&=\PP(\{\omega \in \Omega \colon \|h(\omega)\|_{L^\infty([0,\infty))} \leq A\})
		\geq 1 - \frac{\eta}{2}
	\end{align*} 
	for every $c \in \ell^2 \setminus\{0\}$, where the first identity 
	follows from the self-similarity 
	of Brownian motions as in~\eqref{h-prob}. 
	Now let $\delta = \delta(A)$ be as in~\eqref{eq:AssDeltaEner} 
	and $E>1$ be such that 
	$\|X_0\|_{H^1(\R^d)} \leq E - 1$. 
	In view of 	Lemma~\ref{lem:GeomBrDec}, 
	we infer that there exists $c_0 > 0$ such that 
	for any $\|c\| \geq c_0$, 
	\begin{align*}
		\PP
		\bigg(\bigg\{\omega \in \Omega \colon \|h_c(\omega)\|_{L^\infty([\|c\|^{-1},\infty))} > (4C (2CE)^{\frac{4}{d-2}})^{-1}
		\bigg\}
		\bigg) \leq \frac{\eta}{2},  
	\end{align*}   
	where $C$ is the maximum of the generic constants from Lemmas~\ref{la:SmallIniEner} and~\ref{la:SmallPerturEner}. 
	Hence, the event
	\begin{align*}
		\Omega_c := &\{\omega \in \Omega \colon \|h_c(\omega)\|_{L^\infty([0,\infty))} \leq A\} 
		\\
		&\cap \{\omega \in \Omega \colon 
		4C (2CE)^{\frac{4}{d-2}} \|h_c(\omega)\|_{L^\infty([\|c\|^{-1},\infty))}  \leq 1 \}
	\end{align*}
	has the high probability
	\begin{align*}
		\PP(\Omega_c) \geq 1 - \eta
	\end{align*}
	for all $\|c\| \geq c_0$. 
	
	Furthermore, by Strichartz estimates and the dominated convergence theorem, 
	we have
	\begin{align*}
		\| e^{\imu t \Delta} X_0 \|_{S^1([0,\|c\|^{-1}])} \longrightarrow 0, 
		\quad {\rm as}\ 
		\|c\| \rightarrow \infty.
	\end{align*} 
	This yields that for any $\delta > 0$ so small that 
	\begin{align*} 
		4CA(2\delta)^{\frac{4}{d-2}} \leq 1, 
	\end{align*}  
	there exists $c_0'>0$ such that for any $\|c\| \geq c'_0$ 
	\begin{align*}
		\| e^{\imu t \Delta} X_0 \|_{S^1([0,\|c\|^{-1}])} \leq \delta. 
	\end{align*} 
	
	Consequently, for any $\|c\| \geq \max\{c_0, c_0'\}$ and any $\omega \in \Omega_c$,  
	the conditions of Lemma~\ref{la:SmallIniEner} are satisfied on the interval $[0,\|c\|^{-1}]$ with $h =  h_c(\omega)$. We thus obtain a solution $u_1 \in C([0,\|c\|^{-1}]; H^1(\R^d))$ of~\eqref{hNLS-Energy} with initial value $X_0$ at time~$0$ which satisfies $\|u_1(\|c\|^{-1})\|_{H^1(\R^d)} \leq \|X_0\|_{H^1(\R^d)} + 1 \leq E$. Hence, for any $\omega \in \Omega_c$, the assumptions of Lemma~\ref{la:SmallPerturEner} are satisfied on the interval $[\|c\|^{-1},\infty)$ with $h = h_c(\omega)$ and initial value $u_1(\|c\|^{-1})$. This lemma thus yields a solution $u_2\in C([\|c\|^{-1}, \infty); H^1(\R^d))$.
	Gluing these two solutions together by setting 
	\begin{align*}
		u(t):=u_1(t)\one_{[0,\|c\|^{-1}]}(t) + u_2(t) \one_{(\|c\|^{-1},\infty)}(t),
	\end{align*} 
	we obtain that $u$ solves~\eqref{hNLS-Energy} on $[0,\infty)$, 	$u\in C([0,\infty); H^1(\R^d))$, and 
	$u$ scatters at infinity. The uniqueness properties of $u_1$ and $u_2$ imply uniqueness of $u$ in $S^1([0,\infty))$.
	Therefore, the proof of Theorem \ref{thm:RegNoiseEnerCrit} is complete. 
	\hfill $\square$

	\section{The mass-critical case} \label{Sec-Mass}
	
	Now, we turn to the 
	SNLS \eqref{eq:StochasticNLS}  
	in the mass-critical case $\alpha =1+4/d$ with $d\geq 1$. 
	As in Section \ref{Sec-Energy}, it suffices to consider the rescaled random NLS. 
	
	\begin{lemma}
		\label{lem:MassCritDetPart}
		Let $u_0 \in L^2(\R^d)$ and $h \in L^\infty([0,\infty))$. Take $A, M> 0$ such that $\|u_0\|_{L^2(\R^d)} \leq M$ and $\|h\|_{L^\infty([0,\infty))} \leq A$. Let $\delta = \delta(A) > 0$ 
		be so small that
		\begin{equation}
			\label{eq:Assdelta}
			2^{2 + \frac{4}{d}} C A \delta^{\frac{4}{d}} \leq 1,
		\end{equation}
		where $C$ is a fixed large generic constant. 
		Assume that there exists $T \in (0,\infty)$ such that
		\begin{align}  \label{initial-small-L2}
			\|e^{\imu t \Delta} u_0\|_{L^{\frac{2(d+2)}{d}}([0,T] \times \R^d)} \leq \delta
		\end{align} 
		and 
		\begin{align} \label{h-bdd-L2} 
			(2C+1)^{1 + \frac{4}{d}} 2C \|h\|_{L^\infty([T,\infty))} M^{\frac{4}{d}} \leq 1.
		\end{align} 
		Then there exists a unique solution $u \in C([0,\infty);L^2(\R^d)) \cap L^{\frac{2(d+2)}{d}}([0,\infty) \times \R^d)$ to 
		the initial value problem 
		\begin{equation}\label{hNLS-Mass}
			\left\{\aligned
			\imu \partial_t u + \Delta u &= \lambda h |u|^{\frac{4}{d}}u, \\
			u(0)&=u_0, 		
			\endaligned
			\right.
		\end{equation}  
		and $u$ scatters as $t \rightarrow \infty$, i.e., there exists $u_+ \in L^2(\R^d)$ such that
		\begin{align*}
			\| u(t) - e^{\imu t \Delta} u_+ \|_{L^2(\R^d)} \longrightarrow 0
		\end{align*}
		as $t \rightarrow \infty$.
	\end{lemma}

	\begin{proof}
		We solve~\eqref{hNLS-Mass} in the following two-step procedure, where we first solve the problem in the small-time regime $I:=[0,T]$ and then in the large-time regime $J:=[T,\infty)$.
		
		\medskip 
		{\bf $(i)$ The small-time regime:} 
		We consider the operator 
		\begin{align*}
			\Phi(u)(t) := e^{\imu t \Delta} u_0 - \imu \lambda \int_0^t e^{\imu(t-s)\Delta} (h |u|^{\frac{4}{d}} u)(s) \dd s, 
			\quad t\in I,
		\end{align*} 
		on the closed ball 
		\begin{align*}
			B_{2\delta}(I) := \{u \in L^{\frac{2(d+2)}{d}}(I \times \R^d) \colon \|u\|_{L^{\frac{2(d+2)}{d}}(I \times \R^d)} \leq 2\delta \}.
		\end{align*} 
		Note that the pair $(q,p)$ with $q = p = \frac{2(d+2)}{d}$ is Schr{\"o}dinger admissible with $q' = p' = \frac{2(d+2)}{d+4}$.
		Let 
		\begin{align*}
			S(I) := L^{\frac{2(d+2)}{d}}(I \times \R^d) \qquad \text{and} \qquad 
			N(I) := L^{\frac{2(d+2)}{d+4}}(I \times \R^d). 
		\end{align*}

		An application of Strichartz estimates in Lemma \ref{lem:StrichartzEstimates}  yields 
		that for any $u\in B_{2\delta}(I)$, 
		\begin{align}
			\| \Phi(u) \|_{S(I)} 
			&\leq \| e^{\imu t \Delta} u_0  \|_{S(I)} + C \| h |u|^{\frac{4}{d}} u  \|_{N(I)}
			\leq \delta + C \|h\|_{L^\infty(I)} \| u \|_{S(I)}^{\frac{d+4}{d}}  \nonumber \\ 
			& \leq \delta + C A (2 \delta)^{1 + \frac{4}{d}} \leq 2 \delta,  \label{eq:EstSelfMapMassCrit}
		\end{align} 
		where the last step follows from the smallness condition~\eqref{eq:Assdelta} for $\delta$. 
		Moreover, for any $u,v\in B_{2\delta}(I)$, we have 
		\begin{align}
			\| \Phi(u) - \Phi(v)  \|_{S(I)} &\leq C \| h |u|^{\frac{4}{d}} u - h |v|^{\frac{4}{d}} v  \|_{N(I)}  \nonumber \\ 
			&\leq C \| h \|_{L^\infty(I)} (\| u  \|_{S(I)}^{\frac{4}{d}} + \| v  \|_{S(I)}^{\frac{4}{d}} ) \| u - v  \|_{S(I)} \nonumber\\
			&\leq C A \cdot 2 (2\delta)^{\frac{4}{d}} \| u - v \|_{S(I)} \leq \frac{1}{2} \| u - v \|_{S(I)}, \label{eq:EstContractionMassCrit}
		\end{align}
		where we used a difference estimate analogous to~\eqref{eq:DifferenceEstPointwise} and H\"older's inequality for the second estimate. 
		It follows that $\Phi$ is a contractive self-mapping on $B_{2\delta}(I)$   
		and thus has a unique fixed point $u \in B_{2\delta}(I)$ which solves~\eqref{hNLS-Mass}.
		Another application of Strichartz estimates also shows that  
		$u$ belongs to $C(I;  L^2(\R^d))$. Uniqueness of $u$ in $C(I; L^2(\R^d)) \cap S(I)$ follows from standard arguments.
		
		Finally, using Strichartz estimates and estimating as in~\eqref{eq:EstSelfMapMassCrit}, we infer that 
		\begin{equation}
			\label{eq:EstIniValStep2}
			\| u(T) \|_{L^2(\R^d)} \leq \| u \|_{L^\infty(I;  L^2(\R^d))} \leq \| u_0 \|_{L^2(\R^d)} + C \| h |u|^{\frac{4}{d}} u  \|_{N(I)} \leq M + \delta \leq 2 M,
		\end{equation}
		where we assumed without loss of generality $\delta \leq M$.

		\medskip 
		{\bf $(ii)$ The large-time regime:}  
		Now we construct a solution of~\eqref{hNLS-Mass} on $J= [T, \infty)$ with the initial value $u(T)$ at $T$. To that purpose, we consider the  operator 
		\begin{align*}
			\Psi(v)(t) = e^{\imu (t - T) \Delta} u(T) - \imu \lambda \int_{T}^t e^{\imu (t-s) \Delta} (h |v|^{\frac{4}{d}} v)(s) \dd s, 
			\quad 
			t\geq T, 
		\end{align*}
		on the closed ball 
		\begin{align*}
			B_{M}(J)
			:= \{u \in L^{\frac{2(d+2)}{d}}(J \times \R^d) \colon \|u\|_{L^{\frac{2(d+2)}{d}}(J \times \R^d)} \leq (2C+1)M\}. 
		\end{align*} 
		
		Using Strichartz estimates and~\eqref{eq:EstIniValStep2}, we infer that 
		\begin{align*}
			\| \Psi(v) \|_{S(J)} &\leq \| e^{\imu (t-T) \Delta} u(T)  \|_{S(J)} + C \| h |v|^{\frac{4}{d}} v  \|_{N(J)} \\
			&\leq 
			C \|u(T)\|_{L^2(\R^d)} 
			+ C \|h\|_{L^\infty(J)} \| v \|_{S(J)}^{\frac{d+4}{d}} \\ 
			&\leq 2M C + C \|h\|_{L^\infty(J)} [(2C+1)M]^{1 + \frac{4}{d}}
			\leq (2C+1)M \label{eq:EstSelfMapMassCrit}, 
		\end{align*} 
		where the last step is due to \eqref{h-bdd-L2}. 
		Moreover, 
		for any $v,w \in B_{M}(J)$,
		\begin{align*}
			\| \Psi(v) - \Psi(w)  \|_{S(J)} 
			&\leq C \| h \|_{L^\infty(J)} (\| v  \|_{S(J)}^{\frac{4}{d}} + \| w  \|_{S(J)}^{\frac{4}{d}} ) \| v - w  \|_{S(J)} \nonumber\\
			&\leq C \| h \|_{L^\infty(J)} \cdot 2 [(2C+1)M]^{\frac{4}{d}} \| v - w \|_{S(J)} \leq \frac{1}{2} \| v - w \|_{S(J)}. 
		\end{align*}
		
		Consequently, $\Psi$ is a contractive self-map on $B_{M}(J)$ so that
		there exists a unique fixed point $v$ in $B_{M}(J)$ which solves~\eqref{hNLS-Mass}.
		As in step $(i)$, we conclude that $v$ is the unique solution to \eqref{hNLS-Mass} in $C(J; L^2(\R^d)) \cap S(J)$.
		
		\medskip 
		{\bf $(iii)$}  
		Finally, we glue the two solutions together by setting 
		\begin{align*}
			w(t) = \begin{cases}
				u(t), \quad &t \in [0,T], \\
				v(t), &t \in [T,\infty).
			\end{cases}
		\end{align*}
		Then, $w \in C([0,\infty);  L^2(\R^d)) \cap L^{\frac{2(d+2)}{d}}([0,\infty) \times \R^d)$ and it uniquely solves~\eqref{hNLS-Mass} on $[0,\infty)$. 
		
		Regarding the scattering property, 
		we note that for any $0 < t_1 < t_2<\infty$, 
		\begin{align*}
			\| e^{-\imu t_2 \Delta} w(t_2) - e^{-\imu t_1 \Delta} w(t_1) \|_{L^2(\mathbb{R}^d)} &\leq C \| h |w|^{\frac{4}{d}} w \|_{N(t_1, t_2)} \\
			&\leq C A \| w \|_{L^{\frac{2(d+2)}{d}}((t_1, t_2) \times \R^d)} \longrightarrow 0
		\end{align*}
		as $t_1, t_2 \rightarrow \infty$,  
		since $w \in L^{\frac{2(d+2)}{d}}([0,\infty) \times \R^d)$. 
		This yields the existence of the limit 
		$w_+ := \lim_{t \rightarrow \infty} e^{-\imu t \Delta} w(t)$ in $L^2(\R^d)$,  
		thereby finishing the proof. 
	\end{proof}

	We can now prove the regularization effect of noise on scattering for the stochastic mass-critical NLS.
	
	\medskip 
	{\bf Proof of Theorem~\ref{thm:RegNoiseMassCrit}:} 
	The proof proceeds in an analogous manner as the proof of Theorem~\ref{thm:RegNoiseEnerCrit}.  
	In fact, by virtue of Lemma~\ref{lem:GeomBrDec}, we infer that for any $\eta\in (0,1)$ there exist $A > 0$ and $c_0 > 0$ large enough such that for any  $\|c\| \geq c_0$, the event 
	\begin{align*} 
		\Omega_{c}:=\bigg\{\omega\in \Omega: \|h_c(\omega)\|_{L^\infty([0, \infty))}\leq A, \quad (2C+1)^{1 + \frac{4}{d}} 2C \|h\|_{L^\infty([\|c\|^{-1},\infty))} M^{\frac{4}{d}}\leq 1 
		\bigg\} 
	\end{align*}  
	has the high probability 
	\begin{align*}  
		\P (\Omega_{c}) \geq 1-\eta. 
	\end{align*} 
	Moreover, for $\delta=\delta(A) > 0$ satisfying~\eqref{eq:Assdelta}, 
	we can take $\|c\|$ possibly larger such that 
	\begin{align*}
		\| e^{\imu t \Delta}X_0\|_{S([0,\|c\|^{-1}])} \leq \delta. 
	\end{align*}
	
	Therefore, by virtue of Lemma~\ref{lem:MassCritDetPart},   
	we infer that there exists $c_1 > 0$ such that for any $\|c\| \geq c_1$ and any $\omega \in \Omega_{c}$, the rescaled random NLS \eqref{hNLS-Mass} has a unique solution $u$ on $[0,\infty)$ scattering at infinity.  
	\hfill \qed

	\section*{Acknowledgements} 
	We thank the anonymous reviewer for the careful reading and valuable comments.
    Funded by the Deutsche Forschungsgemeinschaft (DFG, German Research Foundation) -- Project-ID 317210226 -- SFB 1283. 
	D.\ Zhang is also grateful for the NSFC grants (No. 12271352, 12322108) 
	and Shanghai Frontiers Science Center of Modern Analysis.

	
	\bibliographystyle{abbrv}
	\bibliography{RegNoiseEnergyCriticalNLS}
	
\end{document}